\documentclass[12pt]{amsart}
\usepackage{amssymb,latexsym}
\usepackage{enumerate}
\usepackage{hyperref}
\usepackage{fullpage}

\makeatletter \@namedef{subjclassname@2010}{%
  \textup{2010} Mathematics Subject Classification}
\makeatother

\newcounter{thm} \numberwithin{thm}{section}
\newtheorem{Theorem}[thm]{Theorem}

\newtheorem{Lemma}[thm]{Lemma}
\newtheorem{Corollary}[thm]{Corollary}
\newtheorem{Conjecture}[thm]{Conjecture}

\newcommand{\NN}[0]{\mathbb N}

	\newcommand{\QQ}[0]{\mathbb Q}
	\newcommand{\RR}[0]{\mathbb R}

	\newcommand{\ZZ}[0]{\mathbb Z}

	\newcommand{\fS}[0]{\mathfrak S}

\newcommand{\eps}[0]{\varepsilon}

\renewcommand{\mod}[1]{\ (\text{mod }#1)}

\newcommand{\lr}[1]{\left(#1\right)}

\begin{document}


\baselineskip=17pt



\title{Additive Correlation and the Inverse Problem for the Large Sieve}

\author[Brandon Hanson]{Brandon Hanson} \address{Pennsylvania State University\\
University Park, PA}
\email{bwh5339@psu.edu}
\date{}
\begin{abstract}
Let $A\subset [1,N]$ be a set of positive integers with $|A|\gg \sqrt N$. We show that if avoids about $p/2$ residue classes modulo $p$ for each prime $p$, the $A$ must correlate additively with the squares $S=\{n^2:1\leq n\leq \sqrt N\}$, in the sense that we have the additive energy estimate
\[E(A,S)\gg N\log N.\] This is, in a sense, optimal.
\end{abstract}
\maketitle
\section{Introduction}
Let $A\subset [1,N]$ be a set of positive integers. For a positive integer $v$, let \[A_v=\{a\mod v: a\in A\}\subseteq \ZZ/v\ZZ\] denote the set of residue classes covered by $A$ modulo $v$, and let \[A(v;h)=\{a\in A: a\equiv h\mod p\}\subseteq A\] denote the set of elements of $A$ congruent to $h$ modulo $v$. The Large Sieve of Linnik (and developed further by others, notably Montgomery \cite{M}) and the Larger Sieve of Gallagher are useful tools when trying to estimate the cardinality of sets $A$ for which $|A_p|$ is small for primes $p$, see for instance \cite[Chapter 9]{FI}. In fact, under the hypothesis that for each $p$, $|A_p|\leq \alpha p$ the Larger Sieve tells us the remarkably strong result $|A|\ll N^\alpha$ in a fairly elementary way.  When $\alpha$ is about $1/2$, the estimate $|A|\ll N^{1/2}$ is sharp. This can be seen by taking $A$ to consist of the perfect squares up to $N$, or any other dense quadratic sequence. The quadratic case of the ``Inverse Conjecture for the Large Sieve" states that such quadratic sequences are the unique extremizers. Here is a fairly crude form of such the conjecture.
\begin{Conjecture}[Quadratic Inverse Large Sieve]\label{Conjecture}
\footnote{In this article we make frequent use of the asymptotic notation $X\ll Y$ which means that $|X|\leq c|Y|$ for some constant $c$, or equivalently $X=O(Y)$. We also mean }
Let $A\subset [1,N]$ be a set of integers such that $|A_p|\leq (p+1)/2$ for each $p$ and $|A|\gg N^{1/2}$. Then there is a subset $A'\subseteq A$ of size $|A'|\geq \frac{9}{10}|A|$ and such that $A'\subseteq q(\ZZ)$ for some quadratic polynomial $q(x)\in\QQ[x]$.
\end{Conjecture}
\noindent Further reading about this conjecture can be found in \cite{CL, GH, HV, W1, W2}.

Let $X$ and $Y$ be two finite sets of integers. The additive energy between $X$ and $Y$ is the quantity
\[E(X,Y)=|\{(x_1,x_2,y_1,y_2)\in X\times X\times Y\times Y:x_1+y_1=x_2+y_2\}|.\]
This is a quantity intimately related with the sumset
\[X+Y=\{x+y:x\in X, y\in Y\}.\]
For context, the trivial estimates for additive energy are
\[|X||Y|\leq E(X,Y)\leq |X||Y|\min\{|X|,|Y|\}.\] We will write $r_{X+Y}(n)$ and $r_{X-Y}(n)$ for the number of solutions to $x+y=n$ and $x-y=n$, respectively, with $x\in X$ and $y\in Y$. A few moments thought reveals the following formulas:
\begin{equation}\label{Energy1}
E(X,Y)=\sum_{n}r_{X+Y}(n)^2
\end{equation} and
\begin{equation}\label{Energy2}
E(X,Y)=\sum_n r_{X-X}(n)r_{Y-Y}(n).
\end{equation}
The main theorem of this article is the following.
\begin{Theorem}\label{MainTheorem}
Suppose $A$ is a set of at least positive integers in the interval $[1,N]$ satisfying the condition
\[|A_p|\leq \frac{p}{2}+\eps(p)\] for all primes $p$ and some uniformly bounded sequence $\eps(p)\geq 0$. Let $S$ denote the set of perfect squares up to $N$. We have
\[E(A,S)\gg |A||S|+|A|^2\log N.\] In particular, if $|A|\gg \sqrt{N}$, then
\[E(A,S)\gg |A||S|\log N.\]
\end{Theorem}
At first glance, a factor of $\log N$ away from the trivial bound appears quite weak, so before proceeding, here are a few remarks about this theorem.

Firstly, a logarithmic factor is non-trivial. This can be observed by considering dense Sidon subsets of $[1,N]$. Recall a set $x$ is called Sidon if all of its sums are distinct. It is well-known that there are Sidon subsets $X$ of $[1,N]$ of cardinality about $\sqrt N$, the existence of which was proved in \cite{BC}. For such sets $X$, $r_{X-X}(n)\leq 1$ for $n\neq 0$, so that by (\ref{Energy2})
\[E(X,S)=|X||S|+\sum_{n\neq 0}r_{X-X}(n)r_{S-S}(n)\leq |X||Y|+\sum_{n}r_{S-S}(n)=|S|(|X|+|S|)\ll N.\] 
Now if we consider arbitrary subsets $X,Y\subseteq [1,N]$ of integers then $X+Y\subseteq[2,2N]$, we have the obvious estimate $|X+Y|\leq 2N-1$. The Cauchy-Schwarz inequality gives
\[|X||Y|\leq (2N-1)\sum_{2\leq n\leq 2N} r_{X+Y}(n)^2=(2N-1)E(X,Y),\] showing that sets $X$ and $Y$ which are a little larger than $\sqrt N$ in size necessarily have some additive energy between them. So, while sets at the $\sqrt N$-threshold need not have any substantial additive correlation, they only just fail to do so. However, the squares and other sets which avoid residue classes are biased in arithmetic progressions, which are after all cosets of subgroups of $\ZZ$, and so this bias hints at a bit of underlying additive structure. 

Secondly, this logarithmic factor is interesting in that it is what should be expected if Conjecture \ref{Conjecture} were to hold, and so is in that sense best possible. Indeed, a well-known theorem of Ramanujan in \cite{R} states
\begin{equation}\label{Ramanujan}
E(S,S)=\sum_{1\leq n\leq N}r_{S+S}(n)^2\sim\frac{1}{4}N\log N.
\end{equation}
By (\ref{Energy2}), Cauchy-Schwarz, and (\ref{Energy1}), we have for any sets of integers $X$ and $Y$,
\begin{equation}\label{EnergyCauchySchwarz}
E(X,Y)^2\leq E(X,X)E(Y,Y).
\end{equation} If we believe Conjecture \ref{Conjecture}, then $A$ should look like a quadratic sequence $a=sx_a^2+t$ for rational $s$ and $t$ (by completing the square) - so $A$ is an affine transform of the squares. Inequality \ref{EnergyCauchySchwarz} shows that for $u\neq 0$ and $v$ arbitrary
\[E(X,uX+v)^2\leq E(X,X)E(uX+v,uX+v)=E(X,X)^2\]
so that the logarithmic factor of Theorem \ref{MainTheorem} is the most one could hope prove.

Finally, we mention that the inequality (\ref{EnergyCauchySchwarz}) also shows that among all sets $Y$ with comparable of additive energy to $X$, $X$ is essentially the best ``additive partner" for itself in that $E(X,Y)$ is maximized when $Y=X$, provided $Y$ has additive structure comparable to $X$'s. From this, Theorem \ref{MainTheorem} is supporting Conjecture \ref{Conjecture}.

One corollary from Theorem \ref{MainTheorem} is that, at the very least, there are at least $\log N$ elements of $A$ in the image of a single quadratic. Thus we have proved an $\omega$-result. This is far weaker than Conjecture \ref{Conjecture} would give, but it is a positive result in this direction.
\begin{Corollary}
Suppose $A$ is a set of positive integers in the interval $[1,N]$ satisfying the condition
\[|A_p|\leq \frac{p}{2}+\eps(p)\] for all primes $p$ and some uniformly bounded sequence $\eps(p)\geq 0$. If $|A|\gg \sqrt N$, then there is a rational quadratic $q(x)\in \QQ[x]$ such that $|q(\ZZ)\cap A|\gg \log N$.
\end{Corollary}
\begin{proof}
Since $E(A,S)\gg |A||S|\log N$, we have by (\ref{Energy1}),
\[\log N|A||S|\ll\sum_n r_{A+S}(n)^2\leq|A||S|\max_n r(n)\] whence there are at least $c\log N$ solutions $x_a$ to
\[a=n-x_a^2\] with $a\in A$ for some $c>0$. The, in fact integral, quadratic $q(x)=n-x^2$ will suffice.
\end{proof}

Finally, it is known that dense Sidon subsets of $[1,N]$ are well-distributed modulo primes. This has been asserted in \cite{Li} and \cite{K}. We recover a similar such result here. 
\begin{Corollary}
Suppose $A$ is a Sidon set of positive integers in the interval $[1,N]$ satisfying the condition
\[|A_p|\leq \frac{p}{2}+\eps(p)\] for all primes $p$ and some uniformly bounded sequence $\eps(p)\geq 0$. Then \[|A|\ll\frac{\sqrt N}{\log N}.\]
\end{Corollary}
\begin{proof}
By Theorem \ref{MainTheorem} we have
\[E(A,S)\gg |A|(|A|+\sqrt N)\log N.\] On the other hand
\[E(A,S)\leq |A||S|+\sum_{n\neq 0}r_{A-A}(n)r_{S-S}(n)\leq (|A|+\sqrt N)\sqrt N\] since $A$ is Sidon. Rearranging gives the corollary.
\end{proof}
\section{Lemmas and Proofs}
One of the main observations is that one can do a little bit better than the Larger Sieve by considering composite moduli.
\begin{Lemma}\label{CompositeModuli}
Suppose $A$ is a set of integers such that $|A_p|\leq \frac{p}{2}+\eps(p)$ for each prime $p$. Let $\Delta:\NN\to \RR$ be the multiplicative function such that 
\[\Delta(p^k)=p^{k-1}\lr{\frac{p}{2}+\eps(p)}.\] Then for any $v\geq 1$, we have
\[\frac{|A|^2}{\Delta(v)}\leq \sum_{h\mod v}|A(v;h)|^2.\]
\end{Lemma}
\begin{proof}
Since $|A_p|\leq \frac{p}{2}+\eps(p)$, then we also have $|A_{p^k}|\leq p^{k-1}\lr{\frac{p}{2}+\eps(p)}$. By the Chinese Remainder Theorem, it follows that $|A_v|\leq \Delta(v)$. By Cauchy-Schwarz,
\begin{align*}
|A|^2=\lr{\sum_{h\mod v}|A(v;h)|}^2&\leq |A_v|\sum_{h\mod v}|A(v;h)|^2\\
&\leq\Delta(v)\sum_{h\mod v}|A(v;h)|^2.
\end{align*}
\end{proof}

In addition, we will use an estimate for averages of multiplicative functions. This particular estimate can be found in \cite[Appendix A]{FI} as Corollary A.6.
\begin{Lemma}\label{FI}
Let $g:\NN\to \RR$ be a multiplicative function supported on square-free integers. Suppose \[g(p)=\frac{k}{p}+O\lr{\frac{1}{p^2}}\] for some integer $k\geq 1$. Then
\[\sum_{n\leq x}g(m)=\frac{\fS_g}{k!}(\log x)^k+O\lr{(\log x)^{k-1}},\]
where \[\fS_g=\prod_p\lr{1-\frac{1}{p}}^k(1+g(p)).\]
\end{Lemma}

\begin{Corollary}\label{AverageDelta}
Let $\Delta:\NN\to \RR$ be the multiplicative function such that 
\[\Delta(p)=\frac{p}{2}+\eps(p)\] for all primes $p$ and some uniformly bounded sequence $\eps(p)\geq 0$. Then,
\[\sum_{n\leq x}\frac{n^2}{\Delta(n)}\gg x^2\log x.\]
\end{Corollary}
\begin{proof}
Since $\Delta$ is non-negative, we get a lower bound by summing over only those $n$ which are square-free. We have \[\Delta(p)^{-1}=\frac{2}{p}-\frac{4\eps(p)}{p(p+2\eps(p))}=\frac{2}{p}+O\lr{\frac{1}{p^2}}\]
so that if \[M(x)=\sum_{n\leq x}\frac{\mu(n)^2}{\Delta(n)}\] then Lemma \ref{FI} gives
\[M(x)=\frac{\fS_\Delta}{2}(\log x)^2+O\lr{\log x}.\]
Here, \[\fS_\Delta=\prod_p\lr{1-\frac{1}{p}}^2\lr{1+\frac{2}{p}+O\lr{\frac{1}{p^2}}}=\prod_p\lr{1+O\lr{\frac{1}{p^2}}}\] is a convergent product.
Again, since $\Delta(n)$ is non-negative, we are free to include only the terms with $n\geq x/C$ for some $C>1$. Thus
\begin{align*}
\sum_{n\leq x}\frac{n^2\mu(n)^2}{\Delta(n)}&\geq \sum_{x/C\leq n\leq x}\frac{n^2\mu(n)^2}{\Delta(n)}\\
&\geq \lr{x/C}^2\lr{M(x)-M(x/C)}\\
&\geq \lr{x/C}^2\log x\lr{\fS_\Delta\log C+O(1)}.
\end{align*}
Since $\fS_\Delta$ converges, if $C$ is large enough then $\fS_\Delta\log C+O(1)\geq 1$, and the result is proved.
\end{proof}

\begin{Lemma}\label{DivisorsLowerBound}
Let $A\subset[1,N]$ be a set of integers with \[|A_p|\leq\frac{p}{2}+\eps(p)\] for all primes $p$ and some uniformly bounded sequence $\eps(p)\geq 0$. Then
\[\sum_{1\leq u<v\leq \sqrt N}r_{A-A}(uv)\gg|A|^2\log N.\]
\end{Lemma}
\begin{proof}
Observe that $a-b=uv$ with $1\leq u<v$ if and only if $a\equiv b\mod v$ and $b\leq a< b+v^2$, so that
\begin{equation}\label{Identity}
\sum_{1\leq u<v\leq \sqrt N}r_{A-A}(uv)=\sum_{1\leq v\leq \sqrt N}\sum_{\substack{a,b\in A\\ a\equiv b\mod v\\b\leq a< b+v^2}}1.
\end{equation}
For fixed $v$, we divide the interval $[1,N]$ into $J_v=[N/v^2]$ intervals $I_j=[jv^2,(j+1)v^2)$ of length $v^2$ (the last interval may be shorter). If $a,b$ are elements of $A\cap I_j$ with $b\leq a$ and $a$ congruent $b$ modulo $v$, then they are counted in the inner sum of (\ref{Identity}). Thus, we produce a partition of $A$ into sets $A_j=A\cap I_j$ and deduce that
\begin{align*}
\sum_{1\leq u<v\leq \sqrt N}r_{A-A}(uv)&\gg \sum_{1\leq v\leq \sqrt N}\sum_{1\leq j\leq J_v}\sum_{h\mod v}|A_j(v;h)|^2\\
&\geq\sum_{1\leq v\leq \sqrt N}\frac{1}{\Delta(v)}\sum_{1\leq j\leq J_v}|A_j|^2,
\end{align*}
having applied Lemma \ref{CompositeModuli} to each of the sets $A_j$.
By Cauchy-Schwarz,
\[\sum_{1\leq j\leq J_v}|A_j|^2\geq\frac{|A|^2}{J_v}\geq \frac{v^2|A|^2}{N}\] so we have deduced
\[\sum_{1\leq u<v\leq \sqrt N}r_{A-A}(uv)\gg\frac{|A|^2}{N}\sum_{1\leq v\leq \sqrt N}\frac{v^2}{\Delta(v)}.\]
The lemma now follows from Corollary \ref{AverageDelta}.
\end{proof}

\begin{proof}[Proof of Theorem \ref{MainTheorem}]
By passing to a subset of $A$, we may assume that all elements of $A$ lie in a single congruence class modulo $4$. The cost of this is that $|A|$ decreases by a factor of at most $4$. The additive energy between $A$ and $S$ is
\begin{align*}
E(A,S)&=\sum_{m^2,n^2\in S}r_{A-A}(n^2-m^2)\\
&=|A||S|+2\sum_{1\leq m^2<n^2\leq N}r_{A-A}((n-m)(n+m)),
\end{align*}
where in the second line we have extracted the contribution from $r_{A-A}(0)$ and used that $r_{A-A}(k)=r_{A-A}(-k)$.
Now, for $k\neq 0$, $k=n^2-m^2$ has a solution if and only if $k\neq 2\mod 4$, and its solutions are indexed by pairs $(u,v)$ where $uv=k$, $u<v$, and $u\equiv v\mod 2$. This can be seen by considering the system
\[n-m=u,\ n+m=v\] which has as its solution
\[n=\frac{v+u}{2},\ m=\frac{v-u}{2}.\]
Thus
\[E(A,S)=|A||S|+2\sum_{\substack{1\leq u<v\leq \sqrt N\\u\equiv v\mod 2}}r_{A-A}(uv).\]
Because all elements of $A$ are in a single congruence class modulo $4$, $A-A$ is supported only on numbers which are divisible by $4$. Thus if $a-b=n=uv$ for some $a,b\in A$, either $u\equiv v\mod 2$ or else one of $(u/2,2v)$ or $(2u,v/2)$ is going to satisfy the necessary congruence condition. Thus, by reducing our count of pairs $(u,v)$ by a factor of at most $4$, we can remove the condition $u\equiv v\mod 2$, and we obtain a lower bound 
\[E(A,S)\geq \frac{1}{2}\sum_{1\leq u<v\leq \sqrt N/2}r_{A-A}(uv).\] The proof is complete after an application of Lemma \ref{DivisorsLowerBound}.
\end{proof}

\end{document}